\documentclass[oneside, a4paper, 11pt]{article}
\usepackage[hidelinks]{hyperref}
\usepackage{amsmath, amsthm, amssymb, amsfonts}
\usepackage[left=3cm,right=3cm,top=3cm,bottom=3cm,a4paper]{geometry}
\usepackage{thmtools}
\newtheorem{thm}{Theorem}[section]

\newtheorem{lem}[thm]{Lemma}
\newtheorem{conj}[thm]{Conjecture}

\title{\textbf{On a sum of positive rational numbers whose product is 1}}
\author{\large{Jungin Lee}}
\date{}

\begin{document}
\maketitle

\vspace{0mm}

\textbf{Abstract.} In this note, we prove that for every two positive integers $m \geq n \geq 9$, there exist $n$ positive rational numbers whose product is 1 and sum is $m$.

\vspace{5mm}



\vspace{5mm}

\section{Introduction}
Bondarenko [\ref{ref:1}] proved that for every two positive integers $m \geq n \geq 12$, there exist $n$ positive rational numbers whose product is 1 and sum is $m$. The key point of the proof in [\ref{ref:1}] is Lagrange's four square theorem. In this note, we will improve this result to $m \geq n \geq 9$ by using Legendre's three square theorem. Also we will show that for every positive integer $5 \leq m\leq 100$, there exist 5 positive rational numbers whose product is 1 and sum is $m$. 

\section{Results}

\begin{lem} \label{lem2.1}
([\ref{ref:1}], Lemma 1) For every $m>0$, there exist $a,b,c \in \mathbb{Q}_{>0}$ such that $a+b+c=m$ and $abc=1$ if and only if $x^3+y^3+z^3=mxyz$ has a solution in positive integers. 
\end{lem} 

\vspace{3mm}

Let $A_n$ ($n \in \mathbb{N}$) be $\left \{ s \in \mathbb{N} \mid \exists a_1, \cdots, a_n \in \mathbb{Q}_{>0}\: \: \prod_{i=1}^{n}a_i=1, \, \sum_{i=1}^{n}a_i=s \right \}$. Then $n \in A_n$ for every $n \in \mathbb{N}$, and $2+\frac{1}{2}+2+\frac{1}{2}=5$ implies $5 \in A_4$. From the table 3 of [\ref{ref:2}] and Lemma \ref{lem2.1}, $5,6,10,14,30 \in A_3$. Denote $[n]^r$ if $r \in A_n$. It is easy to see that $r_1 \in A_{n_1}$ and $r_2 \in A_{n_2}$ imply $r_1+r_2 \in A_{n_1+n_2}$, so we can define $[n_1]^{r_1}+[n_2]^{r_2}=[n_1+n_2]^{r_1+r_2}$.

\vspace{3mm}

\begin{thm} \label{thm2.2}
For every positive integer $n \geq 9$, $A_n=\left \{ k \in \mathbb{N} \mid k \geq n \right \}$. 
\end{thm} 

\begin{proof} 
\begin{table}[h]
\centering
\small
\begin{tabular}{|c|c|c|c|}\hline
$n$ & $[n]^9$ & $n$ & $[n]^9$ \\\hline
9 & $[9]^9$ & 22 & $[14]^3+[5]^3+[3]^3$ \\\hline
10 & $[5]^4+[5]^5$ & 30 & $[10]^3+[10]^3+[10]^3$ \\\hline
11 & $[5]^3+[6]^6$ & 38 & $[30]^3+[5]^3+[3]^3$ \\\hline
12 & $[6]^3+[6]^6$ & 43 & $[30]^3+[10]^3+[3]^3$ \\\hline
13 & $[6]^3+[5]^4+[2]^2$ & 46 & $[30]^3+[10]^3+[6]^3$ \\\hline
14 & $[6]^3+[5]^3+[3]^3$ & &  \\\hline
\end{tabular}
\caption{remaining cases}
\label{tab:1} 
\end{table}
Suppose that $A_9=\left \{ k \in \mathbb{N} \mid k \geq 9 \right \}$. Then for every $m \geq n \geq 9$, there exist 9 positive rational numbers $b_1, \cdots, b_9$ whose product is 1 and sum is $m-n+9$. In this case, \[a_k=\left\{\begin{matrix} b_k\: \: (1\leq k\leq 9)\\ 1\: \: (9<k\leq n) \end{matrix}\right.\] satisfies $a_1 \cdots a_n=1$ and $a_1+ \cdots +a_n=m$, so $A_n=\left \{ k \in \mathbb{N} \mid k \geq n \right \}$ for every $n \geq 9$. Thus it is enough to show that $A_9=\left \{ k \in \mathbb{N} \mid k \geq 9 \right \}$. \\
For every non-negative integer $a$, $\left \{ x,y,z,m \right \}=\left \{ 2,\,a^2+a+1,\,a^2-a+1,\,a^2+5 \right \}$ and \\
$\left \{ \frac{a^2+147}{4},\, \frac{a^4+6a^3+36a^2+98a+147}{16}, \, \frac{a^4-6a^3+36a^2-98a+147}{16},\, \frac{a^2+33}{2} \right \}$ are solutions of the equation $x^3+y^3+z^3=mxyz$. (These solutions are on the page 205 and 210 of [\ref{ref:3}] as an example, respectively.) $a^4-6a^3+36a^2-98a+147=(a^2-3a)^2+27(a-2)^2+(10a+39)>0$ for every non-negative integer $a$, so Lemma \ref{lem2.1} implies that there exist $p_i,q_i,r_i\in \mathbb{Q}_{>0}$ ($i=1,2$) such that $p_1q_1r_1=p_2q_2r_2=1$, $p_1+q_1+r_1=a^2+5$ and $p_2+q_2+r_2=\frac{a^2+33}{2}$. (If $x^3+y^3+z^3=mxyz$ has a solution in positive rational numbers, then it has a solution in positive integers.) \\
Thus for every non-negative integers $a$, $b$ and $c$, there exist $a_1, \cdots, a_9, b_1, \cdots b_9 \in \mathbb{Q}_{>0}$ such that $a_1 \cdots a_9=b_1 \cdots b_9=1$, $a_1+\cdots +a_9=(a^2+5)+(b^2+5)+(c^2+5)=a^2+b^2+c^2+15$ and $b_1+\cdots +b_9=\frac{a^2+33}{2}+\frac{b^2+33}{2}+\frac{c^2+33}{2}=\frac{a^2+b^2+c^2+99}{2}$. \\
Suppose that a positive integer $m \geq 50$ cannot be represented of the forms $a^2+b^2+c^2+15$ or $\frac{a^2+b^2+c^2+99}{2}$. Then by Legendre's three square theorem, both of $m-15$ and $2m-99$ are of the form $4^k(8t+7)$ ($k,t \geq 0$). $2m-99$ is odd, so $2m-99 \equiv 7 \;\;(mod\;8)$ and this implies $m-15 \equiv 2 \;\;(mod\;4)$, a contradiction. \\
If a positive integer $15 \leq m \leq 49$ cannot be represented of the form $a^2+b^2+c^2+15$, then $m-15=4^k(8t+7)$ for some $k,t \geq 0$. Thus $m-15 \in \left \{ 7,15,23,31,28 \right \}$. Now the only remaining cases are $9 \leq m \leq 14$ and $m \in \left \{ 22,30,38,43,46 \right \}$. Table \ref{tab:1} shows that if $9 \leq m \leq 14$ or $m \in \left \{ 22,30,38,43,46 \right \}$, then $m \in A_9$. 
\end{proof}

\begin{table}[h]
\centering
\small
\begin{tabular}{|c|c|c|c|c|c|}\hline
$n$ & $(b_1,b_2,b_3,b_4,b_5)$ & $n$ & $(b_1,b_2,b_3,b_4,b_5)$ & $n$ & $(b_1,b_2,b_3,b_4,b_5)$ \\\hline
9 & $(1,2,4,18,3)$ & 60 & $(1,28,4,126,3)$ & 84 & $(1,72,66,4,44)$ \\\hline
17 & $(1,2,25,45,30)$ & 61 & $(1,18,14,3,126)$ & 86 & $(1,48,4,150,60)$ \\\hline
25 & $(1,3,5,100,30)$ & 65 & $(1,2,116,18,87)$ & 87 & $(1,20,3,200,30)$ \\\hline
33 & $(1,4,108,150,90)$ & 66 & $(1,2,126,48,28)$ & 88 & $(1,25,2,125,50)$ \\\hline
35 & $(1,2,56,18,84)$ & 67 & $(1,38,12,2,57)$ & 89 & $(1,60,34,3,85)$ \\\hline
41 & $(1,3,108,150,90)$ & 70 & $(1,2,132,18,33)$ & 90 & $(1,75,10,147,14)$ \\\hline
45 & $(1,3,4,162,9)$ & 73 & $(1,7,3,196,42)$ & 93 & $(1,2,180,12,10)$ \\\hline
47 & $(1,3,128,144,6)$ & 75 & $(1,32,4,169,104)$ & 95 & $(1,42,15,2,105)$ \\\hline
49 & $(1,3,25,4,150)$ & 79 & $(1,17,90,3,170)$ & 96 & $(1,63,6,196,42)$ \\\hline
54 & $(1,2,100,180,6)$ & 80 & $(1,6,2,147,14)$ & 97 & $(1,80,12,200,30)$ \\\hline
57 & $(1,3,160,24,10)$ & 81 & $(1,45,156,4,130)$ & 100 & $(1,78,9,192,104)$ \\\hline
59 & $(1,3,160,24,60)$ & 82 & $(1,35,3,140,30)$ & & \\\hline
\end{tabular}
\caption{solution of $\frac{b_2}{b_1}+\frac{b_3}{b_2}+\frac{b_4}{b_3}+\frac{b_5}{b_4}+\frac{b_1}{b_5}=n$ ($n \leq 100$)}
\label{tab:2} 
\end{table}

It is easy to see that $n \in A_4$ implies $n+1 \in A_5$. The list of elements of $A_4$ up to 100 (possibly not complete) is given in 9.9.2 of [\ref{ref:4}]. From this, we obtain that if a positive integer $5 \leq n \leq 100$ satisfies $n \notin A_5$, then $n$ is one of the 9,\,17,\,25,\,33,\,35,\,41,\,45,\,47,\,49,\,54,\,57,\,59,\,60, \\ 61,\,65,\,66,\,67,\,70,\,73,\,75,\,79,\,80,\,81,\,82,\,84,\,86,\,87,\,88,\,89,\,90,\,93,\,95,\,96,\,97 or 100. Table \ref{tab:2} shows that all of these integers are elements of $A_5$. Thus every positive integer $5 \leq n \leq 100$ is an element of $A_5$. This result supports the following conjecture.

\begin{conj} \label{conj2.3}
For every positive integer $n \geq 5$, $A_n=\left \{ k \in \mathbb{N} \mid k \geq n \right \}$. 
\end{conj} 

\vspace{3mm}

Bondarenko ([\ref{ref:1}], Theorem 3) proved that if $m=4k^2$ ($k \in \mathbb{N}$) and $k$ is not divisible by 3, then $m \notin A_3$. We check that $\frac{b_2}{b_1}+\frac{b_3}{b_2}+\frac{b_4}{b_3}+\frac{b_1}{b_4}=8$ does not have a solution where $b_i$ ($1 \leq i \leq 4$) are positive integers and $1\leq b_1 \leq 200, 1 \leq b_2,b_3,b_4 \leq 1000$. From this numerical evidence, we make the following conjecture. 

\begin{conj} \label{conj2.4}
There exists a positive integer $k \geq 4$ such that $k \notin A_4$. In particular, $8 \notin A_4$. 
\end{conj} 

\vspace{3mm}

\textbf{Acknowledgments} \hspace{3mm} The author would like to thank the anonymous referee for their helpful comments. The author also thanks Minryoung Kim who find solutions of the equation $\frac{b_2}{b_1}+\frac{b_3}{b_2}+\frac{b_4}{b_3}+\frac{b_5}{b_4}+\frac{b_1}{b_5}=n$ for some positive integers $n \leq 100$. 

\section{References}

\begin{enumerate}

\item{A. V. Bondarenko, Investigation of one class of Diophantine equations, Ukrainian Math. J. 52(6) (2000) 953-959.} \label{ref:1} 

\item{E. Dofs, On some classes of homogeneous ternary cubic Diophantine equations, Ark. Mat. 13(1) (1975) 29-72.} \label{ref:2} 

\item{E. Dofs, Solutions of $x^3+y^3+z^3=nxyz$, Acta Arith. 73(3) (1995) 201-213.} \label{ref:3}

\item{A. Nowicki, Liczby Wymierne, Podr{\'{o}}{\.{z}}e po Imperium Liczb, cz.1, Wydanie drugie. Wydawnictwo OWSIiZ, Torun, Olszty{\'{n}}, 2012.} \label{ref:4}

\end{enumerate}

\vspace{3mm}

Department of Mathematical Sciences, Seoul National University, Seoul 151-747, Korea 

e-mail: moleculesum@snu.ac.kr

\end{document}